\documentclass{amsart}
\usepackage{amsmath,amssymb,colordvi,verbatim}
\usepackage{graphicx}
\usepackage{anysize}
\newtheorem{theorem}{Theorem}
\newtheorem{corollary}[theorem]{Corollary}

\newtheorem{proposition}[theorem]{Proposition}
\newtheorem{definition}[theorem]{Definition}

\newtheorem{remark}[theorem]{\it Remark}

\newcommand{\Na}{\nabla}
\newcommand{\N}{\mathbb{N}}
\newcommand{\R}{\mathbb{R}}

\begin{document}

\title[Discrete resolvent operators]
%{An explicit representation of discrete fractional resolvent families in Banach spaces.}
{Explicit representation of discrete fractional resolvent families in Banach spaces.}\footnote{This paper is published in Fract. Calc. Appl. Anal. Vol. 24, No 6 (2021), pp. 1853-1878,  DOI: 10.1515/fca-2021-0080. Available at https://www.degruyter.com/journal/key/FCA/html , and since 2022, at  https://link.springer.com/journal/13540/volumes-and-issues.}

\author{Jorge Gonz\'alez-Camus}
\address{Universidad Tecnológica Metropolitana, Chile.}
\email{j.gonzalezc@utem.cl}

\author{Rodrigo Ponce}
\address{Universidad de Talca, Instituto de Matem\'aticas, Casilla 747, Talca-Chile.}
\email{rponce@inst-mat.utalca.cl}
%\thanks{Research was supported by CONICYT (Chile) through PBCT program ACT-56}

%\subjclass[2020]{Primary 47D06, Secondary 34G10, 39A06, 39A12}


%%% ----------------------------------------------------------------------

\begin{abstract}

In this paper we introduce a discrete fractional resolvent family $\{S_{\alpha,\beta}^n\}_{n\in\N_0}$ generated by a closed linear operator in a Banach space $X$ for a given $\alpha,\beta>0.$ Moreover, we study its main properties and, as a consequence, we obtain a method to study the existence and uniqueness of the solutions to discrete fractional difference equations in a Banach space.

\end{abstract}

%%% ----------------------------------------------------------------------
\maketitle
%%% ----------------------------------------------------------------------

\section{Introduction}

During the last six decades, the theory of $C_0-$semigroups of operators on Banach spaces has been used by many authors as a powerful tool to study linear and nonlinear partial differential equations, as well as, to study concrete equations arising in mathematical physics, probability theory, engineering, biological processes, among others. See for instance \cite{En-Na-00}. Typically, in these situations, the problems are modeled by using partial differential equations of first order with unbounded linear operators. However, there are many problems in applied sciences, including, problems in transport dynamics, anomalous diffusion, non-Brownian motion, and many others, where the model of a partial differential equation of first order is not completely satisfactory.

In recent decades, some investigations have demonstrated that some of these phenomena can be described more appropriately by means of time-fractional differential equations, see for instance, \cite{Al-Ca-Va-16,Ca-Pl-15,Ki-Sr-Tr-06,Li-Pe-Ji-12,Li-Li-15,Mi-Ro-93,To-Ta-17}. As the one-parameter semigroups represent the natural framework to study differential equations of first order, in the case of time-fractional differential equations, the theory of {\em continuous} fractional resolvent families of one-parameter (that extends the theory of semigroups) gives one of the main tool to study such equations, see for instance \cite{He-Me-Po-21,Li-Ch-Li-10,Li-Pe-Ji-12,Pon-20a}. For example, if we consider the time-fractional differential equation
\begin{equation}\label{Eq0.A}
\partial_t^\alpha u(t)=Au(t)+f(t),\quad t>0,
\end{equation}
under the initial condition $u(0)=x_0,$ where for  $0<\alpha<1,$ $\partial_t^\alpha$ corresponds to the Caputo fractional derivative, $f$ is a suitable function and $A$ generates an exponentially bounded fractional resolvent family $\{S_{\alpha,\alpha}(t)\}_{t\geq 0}$ (see for instance \cite{Cu-07,He-Me-Po-21,Li-Ch-Li-10,Li-Pe-Ji-12,Pon-20a}). Then, the solution to \eqref{Eq0.A} is given by
\begin{equation*}
u(t)=S_{\alpha,1}(t)x_0+\int_0^t S_{\alpha,\alpha}(t-s)f(s)ds,
\end{equation*}
where $S_{\alpha,1}(t):=\int_0^t \frac{(t-s)^{-\alpha}}{\Gamma(1-\alpha)}S_{\alpha,\alpha}(s)ds.$ This theory of continuous fractional resolvent families has been widely studied in the last recent years. See for instance \cite{Li-FPo-12,Pr,Wa-Ch-Xi-12} and the references therein. However, these continuous-time problems sometimes need to be studied, for practical purposes, as discrete problems.

%On the other hand, in the last decade, the study of existence and qualitative properties of discrete fractional difference equations has been a topic of great interest and there is an extensive literature in this subject, see for instance \cite{At-El-09,Go-11}. However, these articles focus mainly on scalar fractional difference equations. Very recently, C. Lizama in \cite{Li-17} introduced, to the best of our knowledge, the first study on fractional difference equations with unbounded linear operators.

Although the first investigations on difference of fractional order date back to the work of Kuttner \cite{Ku-57}, in the last decade, the study of existence and qualitative properties of discrete fractional difference equations has been a topic of great interest and there is an extensive recent literature in this subject, see for instance \cite{Ab-11,At-El-09,Fe-12,Go-11,Go-Pe-15} and the references therein. However, these articles focus mainly on scalar fractional difference equations. Very recently, C. Lizama in \cite{Li-17} introduced, to the best of our knowledge, the first study on fractional difference equations with unbounded linear operators. Here, the author finds an interesting relation between the existence of solutions to an abstract fractional difference and a discrete family of linear operators that corresponds to a discretization of a continuous fractional resolvent family. More concretely, if $0<\alpha<1,$ $A$ is a closed linear operator defined on a Banach space $X$ and $_C\Delta^\alpha u^n$ is the approximation of the Caputo fractional derivative $\partial^\alpha_t u(t)$ (at time $t=n$) defined by
\begin{equation*}
_C\Delta^{\alpha}u^n:=\sum_{j=0}^n \frac{\Gamma(1-\alpha+n-j)}{\Gamma(1-\alpha)\Gamma(n-j+1)}(u^{j+1}-u^j),
\end{equation*}
where $u^j:=\int_0^\infty p_j(t)u(t)dt$ and $p_j(t):=t^j/j!e^{-t}$ is the Poisson distribution for $j\in\N_0,$ then solution to the fractional difference equation
\begin{eqnarray*}\label{eqDiscrete-0}
_C\Delta^\alpha u^n=Au^{n+1}, \quad n\in\N,
\end{eqnarray*}
is given by $u^n=S_{\alpha,1}^n(I-A)u_0,$ where $u_0\in D(A),$ $\{S_{\alpha,1}(t)\}_{t\geq 0}$ is the continuous resolvent family generated by $A,$ whose Laplace transform satisfies $\hat{S}_{\alpha,1}(\lambda)=\lambda^{\alpha-1}(\lambda^\alpha-A)^{-1}$ and $S_{\alpha,1}^n:=\int_0^\infty p_n(t) S_{\alpha,1}(t)dt,$ for all $n\in\mathbb{N}_0.$ From \cite{Li-Ch-Li-10} it follows that the resolvent family $\{S_{\alpha,1}(t)\}_{t\geq 0}$ satisfies the resolvent equation
\begin{equation*}
S_{\alpha,1}(t)x=x+A\int_0^t g_\alpha(t-s)S_{\alpha,1}(s)xds,\quad x\in X,\,\, t\geq 0,
\end{equation*}
where $g_\alpha(t):=t^{\alpha-1}/\Gamma(\alpha).$ Moreover, from \cite{Li-17}, it is easy to see that the sequence of operators $\{S_{\alpha,1}^n\}_{n\in\N_0}$ verifies a similar relation:
\begin{equation*}
S_{\alpha,1}^nx=x+A\sum_{j=0}^n k^\alpha(n-j)S_{\alpha,\alpha}^jx,\quad x\in X,\, n\in \N_0,
\end{equation*}
where $k^\alpha(j):=\frac{\Gamma(\alpha+j)}{\Gamma(\alpha)\Gamma(j+1)}.$ According to the Poisson distribution, we notice that for each $n\in\N_0,$ $S_{\alpha,1}^n$ corresponds to an approximation of $S_{\alpha,1}(t)$ at time $t=n.$ Similarly, in \cite{Ab-Al-Di-21,Ab-Li-16,Ab-Li-Mi-Ve-19,Al-Di-Li-20,Li-He-Zh-20,Xi-Wa-18} the authors have introduced several discrete resolvent families to study fractional difference equations in Banach spaces. See \cite{Go-Li-20a,Go-Li-20b} for related results. We notice also that, fractional difference equations are closely related with discretization of fractional differential equations in Banach spaces, see for instance \cite{Ji-La-Zh-16,Ji-La-Zh-19,Ji-Li-Zh-18,Ji-Li-Zh-18b,Ji-Li-Zh-19,Lu-86,Pon-19}.

Although the continuous fractional resolvent families are an important tool in the study of fractional differential equations in Banach spaces and, there are many published papers on these families, their properties and applications, there are only a few articles on discrete fractional resolvent families generated by unbounded operators, and therefore, the study of the solutions of discrete fractional difference equations in Banach space has been limited by the lack of this tool.

In this paper, for a given $\alpha,\beta>0$ and a step-size $\tau>0,$ we introduce the general discrete resolvent family $\{S_{\alpha,\beta}^n\}_{n\in\N_0}$ generated by a closed linear operator $A$ in a Banach space $X,$ and we study its main properties. Moreover, we give a method to study the existence and uniqueness of solutions to discrete fractional difference equations in Banach spaces.

%We notice here, for $\tau$ small enough, $S_{\alpha,\beta}^n$ corresponds precisely to an approximation of a continuous resolvent family $S_{\alpha,\beta}(t)$ at $t=n\tau$ for each $n\in\N_0.$

The paper is organized as follows. In Section \ref{Sect2} we give the preliminaries. In Section \ref{Sect3} we introduce the discrete fractional resolvent family $\{S_{\alpha,\beta}^n\}_{n\in\N_0}.$ Moreover, we study its main properties, conditions on the operator $A$ in order to be the generator of $\{S_{\alpha,\beta}^n\}_{n\in\N_0}$ and we show that $\{S_{\alpha,\beta}^n\}_{n\in\N_0}$ can be written as
\begin{equation*}
S_{\alpha,\beta}^nx=\sum_{j=1}^{n+1}a_{n,j}\tau^{-\alpha j}(\tau^{-\alpha}-A)^{-j}x, \quad n\in \N_0,
\end{equation*}
for all $x\in X,$ where $a_{n,j}$ are constants depending on $\alpha,\beta$ and $\tau.$ Finally, as an application of the results given previously, in Section \ref{Sect4} we study the existence and uniqueness of solutions to a fractional difference equation in a Banach space.

\section{Preliminaries}\label{Sect2}

The set of non-negative integer numbers will be denoted by $\mathbb{N}_0$ and the non-negative real numbers by $\mathbb{R}_0^+.$ Take $\tau>0$ fixed and $n\in\mathbb{N}_0.$ We define the positive functions $\rho_n^\tau$ by
\begin{equation*}
\rho_n^\tau(t):=e^{-\frac{t}{\tau}}\left(\frac{t}{\tau}\right)^{n}\frac{1}{\tau n!},
\end{equation*}
for all $t\geq 0,$ $n\in\mathbb{N}_0.$ An easy computation shows that
\begin{equation*}
\int_0^\infty \rho_n^\tau(t)dt=1, \quad \mbox{ for all }\quad n\in\mathbb{N}_0.
\end{equation*}

For a given Banach space $X,$ $s(\mathbb{N}_0,X)$ denotes the vectorial space consisting of all vector-valued sequences $v:\mathbb{N}_0\to X.$  The {\em backward Euler operator} $\nabla_\tau:s(\mathbb{N}_0,X)\to s(\mathbb{N}_0,X)$ is defined by
%% \Na=\nabla
\begin{equation*}
\Na_\tau v^n:=\frac{v^n-v^{n-1}}{\tau}, \quad n\in\N.
\end{equation*}
For $m\geq 2,$ we define $\Na_\tau^m:s(\mathbb{N}_0,X)\to s(\mathbb{N}_0,X)$ recursively by
\begin{equation*}\label{DefNablam}
(\Na_\tau^m v)^n:=\Na_\tau^{m-1}(\Na_\tau v)^n, \quad n\geq m.
\end{equation*}
Here $\Na_\tau^{1}$ is defined as $\Na_\tau^{1}:=\Na_\tau$ and $\Na_\tau^{0}$ as the identity operator. As in \cite[Chapter 1, Section 1.5]{Go-Pe-15} we define by convention
\begin{equation}\label{EqConvention}
\sum_{j=0}^{-k} v^j=0
\end{equation}
for all $k\in\mathbb{N}.$

The operator $\Na_\tau^m$ is called the {\em backward difference operator of order} $m.$ It is easy to show that if $v\in s(\mathbb{N}_0,X)$ then
\begin{equation*}
(\Na_\tau^m v)^n=\frac{1}{\tau^m}\sum_{j=0}^m {m\choose j} (-1)^{j}v^{n-j}, \quad n\in\N.
\end{equation*}

For a given $\alpha>0,$ define the function $g_\alpha$ as $g_\alpha(t):=\frac{t^{\alpha-1}}{\Gamma(\alpha)}.$ Now, we introduce the following sequence
\begin{equation*}\label{eq2.1}
k^\alpha_\tau(n):=\int_0^\infty \rho_n^\tau(t)g_\alpha(t)dt, \quad n\in\mathbb{N}_0,\, \alpha>0.
\end{equation*}
It is easy to see that
\begin{equation}\label{eq2.2}
k^\alpha_\tau(n)=\frac{\tau^{\alpha-1}\Gamma(\alpha+n)}{\Gamma(\alpha)\Gamma(n+1)}=\frac{\Gamma(\alpha+n)}{\Gamma(n+1)}g_\alpha(\tau), \quad n\in\mathbb{N}_0,\, \alpha>0.
\end{equation}
In particular, we notice that $k^1_\tau(n)=1$ for all $n\in \N_0.$

\begin{definition}\cite{Pon-19} Let $\alpha>0.$ The $\alpha^{\rm th}-$fractional sum of $v\in\mathcal{F}(\R;X)$ is defined by
\begin{equation*}\label{eq0.1}
(\Na^{-\alpha}_\tau v)^n:=\tau\sum_{j=0}^n k^\alpha_\tau(n-j)v^j,\quad n\in\N_0.
\end{equation*}
\end{definition}

\begin{definition}\cite{Pon-19} Let $\alpha\in\mathbb{R}_+\setminus \N_0.$
The {\em Caputo fractional backward difference operator of order $\alpha$,} $_C\Na^\alpha: \mathcal{F}(\R_+;X)\to \mathcal{F}(\R_+;X),$ is defined by
\begin{equation*}
(_C\Na^\alpha v)^n:=\Na^{-(m-\alpha)}_\tau(\Na^m_\tau v)^n,\quad n\in\N,
\end{equation*}
where $m-1<\alpha<m.$
\end{definition}

For a given $\alpha\in \N_0,$ the fractional backward difference operators $_C\Na^\alpha$ is defined as the backward difference operator $\Na^\alpha_\tau.$ Moreover, if $0<\alpha<1$ and $n\in\N,$ then $_C\Na^{\alpha+1}v^n=\,_C\Na^\alpha(\Na^1 v)^n.$ However, $_C\Na^{\alpha+1}v^n\neq\,_C\Na^1(_C\Na^\alpha v)^n,$ (see \cite[Section 2]{Pon-19}).

For a given discrete sequence of operators $\{S^n\}_{n\in \N_0}\subset \mathcal{B}(X)$ and a scalar sequence $c=(c^n)_{n\in\N_0},$  we define the discrete convolution $c\star S$ as
\begin{equation*}
({c}\star {S})^n:=\sum_{k=0}^n {c}^{n-k}{S}^k, \quad n\in\mathbb{N}_0.
\end{equation*}

Moreover, for scalar valued sequences $b=(b^n)_{n\in\N_0}$ and $c=(c^n)_{n\in\N_0},$ we define $({b}\star {c}\star {S})^n:=({b}\star ({c}\star {S}))^n$ for all $n\in\N_0.$

As in \cite[Corollary 2.9]{Pon-19} we can prove the following convolution property. If $\alpha,\beta>0,$ then

\begin{equation}\label{Semigroup-for-K}
k_\tau^{\alpha+\beta}(n)=\tau\sum_{j=0}^n k^{\alpha}_\tau(n-j)k_\tau^\beta(j),
\end{equation}
for all $n\in\N_0.$ Given $s\in s(\mathbb{N}_0,X),$ its $Z$-transform, $\tilde{s},$ is defined by
\begin{equation*}
\tilde{s}(z):=\sum_{j=0}^\infty z^{-j}s^j,
\end{equation*}
where $s^j:=s(j)$ and $z\in \mathbb{C}.$ We notice that the convergence of this series holds for $|z|>R,$ where $R$ is large enough. It is a well known fact that if $s_1,s_2\in s(\mathbb{N}_0,X)$ and $\tilde{s_1}(z)=\tilde{s_2}(z)$ for all $|z|>R$ for some $R>0,$ then $s_1^j=s_2^j$ for all $j=0,1,...$ Moreover, the $Z$-transform is a linear operator on $ s(\mathbb{N}_0,X)$ and satisfies the finite discrete convolution property (see for instance \cite{Ag-Cu-Li-14}):
\begin{equation}\label{Zconvolution}
\widetilde{s_1\star s_2}(z)=\tilde{s_1}(z)\tilde{s_2}(z), \quad \ s_1,s_2\in s(\mathbb{N}_0,X).
\end{equation}

The operator $A:D(A)\subset X\to X$ is called $\omega${\em -sectorial of angle} $\theta,$ if there exist
$\theta\in [0,\pi/2)$ and $\omega\in \mathbb{R}$ such that its resolvent exists in the sector $\omega+\Sigma_\theta:=\left\{\omega+\lambda: \lambda\in \mathbb{C}, |\arg(\lambda)|<\frac{\pi}{2}+\theta\right\}\setminus\{\omega\}$ and
\begin{equation*}
\|(\lambda-A)^{-1}\|\leq \frac{M}{|\lambda-\omega|},
\end{equation*}
for all $\lambda \in \omega+\Sigma_\theta.$  In case $\omega=0$ we say that $A$ is sectorial of angle $\phi+ \pi/2.$ More details on sectorial operators can be found in \cite{Haa-06}.

\begin{definition}
A family of linear operators $\{S(t)\}_{t\geq 0}\subset \mathcal{B}(X)$  is said to be {\em exponentially bounded}, if there exist constants $M,\,\omega\in \R$ such that $\|S(t)\|\leq Me^{\omega t},$ for all $t\geq 0.$
\end{definition}

\begin{proposition}\label{Prop2.14}
Let $\{S(t)\}_{t\geq 0}\subset \mathcal{B}(X)$ be a family of exponentially bounded linear operators with $\|S(t)\|\leq Me^{\omega t},$ where $M>0$ and $\omega<\frac{1}{\tau}.$ Let $x\in X.$ If we define the sequence ${S}^nx$ for each $n\in \N$ by
\begin{equation*}
S^nx:=\int_0^\infty \rho_n^\tau (t)S(t)xdt.
\end{equation*}
Then
\begin{equation*}
\tilde{S}(z)x=\frac{1}{\tau}\hat{S}\left(\frac{1}{\tau}\left(1-\frac{1}{z}\right)\right)x,
\end{equation*}
for all $|z|>1.$
\end{proposition}
\begin{proof}
The hypothesis implies that
\begin{equation*}
\|S^nx\|\leq M\int_0^\infty \rho_n^\tau (t)e^{\omega t}\|x\|dt=\frac{M}{(1-\omega\tau)^{n+1}}\|x\|,
\end{equation*}
for all $n\in\N_0.$ Therefore, the $Z$-transform of $S$ exists for all $|z|>1.$ On the other hand, the hypothesis implies that the Laplace transform of $S$ exists for all ${\rm Re}(\lambda)>0.$ Thus
\begin{eqnarray*}
\tilde{S}(z)x&=&\sum_{n=0}^\infty z^{-n}S^nx\\
&=&\sum_{n=0}^\infty z^{-n}\int_0^\infty \rho_n^\tau (t)S(t)xdt\\
&=&\int_0^\infty e^{-\frac{t}{\tau}}\sum_{n=0}^\infty z^{-n}\left(\frac{t}{\tau}\right)^n\frac{1}{\tau n!}S(t)xdt\\
&=&\frac{1}{\tau}\int_0^\infty e^{-\frac{t}{\tau}}\sum_{n=0}^\infty \frac{1}{n!}\left(\frac{t}{\tau z}\right)^nS(t)xdt\\
&=&\frac{1}{\tau}\int_0^\infty e^{-\frac{t}{\tau}\left(1-\frac{1}{z}\right)}S(t)xdt\\
&=&\frac{1}{\tau}\hat{S}\left(\frac{1}{\tau}\left(1-\frac{1}{z}\right)\right)x.
\end{eqnarray*}
\end{proof}

We notice that a similar result holds for vector-valud functions. Thus, if $(f^n)_{n\in\N_0}$ denotes the sequence defined by $f^n:=\int_0^\infty \rho_n^\tau (t)f(t)dt$ for a given function $f:\mathbb{R}_+\to X,$ then
\begin{equation*}
\tilde{F}(z)=\frac{1}{\tau}\hat{f}\left(\frac{1}{\tau}\left(1-\frac{1}{z}\right)\right),
\end{equation*}
where $F$ denotes the sequence associated to $(f^n)_{n\in\N_0}.$

\section{Discrete fractional resolvent families}\label{Sect3}

In this Section we introduce the notion of discrete fractional resolvent family generated by a closed linear operator $A$ in a Banach space and we study its main properties.

\begin{definition}\label{DefResolvent}
Let $1\leq \alpha\leq 2$ and $0<\beta\leq 2$ be given. Let $A$ be a closed linear operator defined on a Banach space $X.$ An operator-valued sequence $\{S_{\alpha,\beta}^n\}_{n\in\N_0}\subset B(X)$ is called a discrete $(\alpha,\beta)$-{\em resolvent family} generated by $A$ if it satisfies the following conditions
\begin{enumerate}
  \item $S_{\alpha,\beta}^n\in D(A)$ for all $x\in X$ and $AS_{\alpha,\beta}^nx=S_{\alpha,\beta}^nAx$ for all $x\in D(A),$ and $n\in \N_0.$
  \item For each $x\in X$ and $n\in \N_0,$
  \begin{equation}\label{ResolventEqn}
    S_{\alpha,\beta}^nx=k^\beta_\tau(n)x+\tau A(k^\alpha_\tau \star S_{\alpha,\beta})^nx=k^\beta_\tau(n)x+\tau A\sum_{j=0}^n k^\alpha_\tau(n-j)S_{\alpha,\beta}^jx.
  \end{equation}
\end{enumerate}
\end{definition}

\begin{proposition}\label{Prop2.1}
Let $\{S_{\alpha,\beta}^n\}_{n\in\N_0}\subset B(X)$ be a discrete $(\alpha,\beta)$-resolvent family generated by $A.$ Then,
\begin{enumerate}
  \item $\tau^{-\alpha}\in\rho(A),$ and
  \item For $n=0,$ we have
    \begin{equation*}
     S_{\alpha,\beta}^0=k_\tau^\beta(0)\tau^{-\alpha}\left(\tau^{-\alpha}-A\right)^{-1}=\tau^{\beta-1-\alpha}\left(\tau^{-\alpha}-A\right)^{-1}.
    \end{equation*}
\end{enumerate}
\end{proposition}
\begin{proof}
We notice that, by \eqref{ResolventEqn}, we have
\begin{equation*}
S_{\alpha,\beta}^nx=k^\beta_\tau(n)x+k^\alpha_\tau(0)\tau AS_{\alpha,\beta}^nx+\tau A\sum_{j=0}^{n-1} k^\alpha_\tau(n-j)S_{\alpha,\beta}^jx,\quad \mbox{for all }\,\,x\in X.
\end{equation*}
As $ k^\alpha_\tau(0)\tau=\tau^{\alpha},$ for all $\alpha>0,$ we get (for $n=0$)
\begin{equation*}
   S_{\alpha,\beta}^0x=k_\tau^\beta(0)x+\tau^\alpha AS_{\alpha,\beta}^0x,
\end{equation*}
and hence
\begin{equation*}
  \left(\tau^{-\alpha}-A\right)S_{\alpha,\beta}^0x=k_\tau^{\beta}(0)\tau^{-\alpha}x,
\end{equation*}
for all $x\in X.$ Now, from Definition \ref{DefResolvent} we obtain
\begin{equation*}
  S_{\alpha,\beta}^0\left(\tau^{-\alpha}-A\right)x=\tau^{-\alpha}S_{\alpha,\beta}^0x-S_{\alpha,\beta}^0Ax=\left(\tau^{-\alpha}
  -A\right)S_{\alpha,\beta}^0x=k_\tau^{\beta}(0)\tau^{-\alpha}x,
\end{equation*}
for all $x\in X.$ Since $A$ is a closed linear operator, we conclude that $\tau^{-\alpha}\in \rho(A)$ and
\begin{equation*}
  S_{\alpha,\beta}^0x=k_\tau^{\beta}(0)\tau^{-\alpha}\left(\tau^{-\alpha}-A\right)^{-1}x,
\end{equation*}
for all $x\in X.$
\end{proof}

An easy computation (see Proposition \ref{Prop2.14}) shows that, for a given $\alpha>0,$ the $Z$-transform of the sequence $\{k_\tau^\alpha(n)\}_{n\in\N_0}$ is given by
\begin{equation}\label{Zk}
\widetilde{k_\tau^\alpha}(z)=\tau^{\alpha-1}\frac{z^\alpha}{(z-1)^\alpha}.
\end{equation}

\begin{proposition}\label{Z-Transform}
Let $\{S_{\alpha,\beta}^n\}_{n\in\N_0}\subset B(X)$ be a discrete $(\alpha,\beta)$-resolvent family generated by $A.$ Then, its Z-transform satisfies
\begin{equation*}\label{Eq-Z-Transform}
\widetilde{S_{\alpha,\beta}}(z)x=\frac{1}{\tau}\left(\frac{z-1}{\tau z}\right)^{\alpha-\beta}\left(\left(\frac{z-1}{\tau z}\right)^\alpha-A\right)^{-1}x,
\end{equation*}
for all $x\in X.$
\end{proposition}
\begin{proof}

Using the definition \eqref{ResolventEqn} and the identity \eqref{Zconvolution} we have

\begin{equation*}
\widetilde{S_{\alpha,\beta}}(z)x=\widetilde{k}^\beta_\tau(z)x+\tau\widetilde{(k^\alpha_\tau \star AS_{\alpha,\beta})}(z)x=\widetilde{k}^\beta_\tau(z)x+\tau\widetilde{k}^\alpha_\tau(z)A\widetilde{S_{\alpha,\beta}}(z)x.
\end{equation*}
A straightforward computation and using \eqref{Zk} yield
\begin{equation*}
\widetilde{S_{\alpha,\beta}}(z)x=\frac{1}{\tau}\left(\frac{z-1}{\tau z}\right)^{\alpha-\beta}\left(\left(\frac{z-1}{\tau z}\right)^\alpha-A\right)^{-1}x
\end{equation*}
and the proof is finished.
\end{proof}

The next result gives a functional equation to the discrete fractional resolvent families $\{S_{\alpha,\beta}^n\}_{n\in\N_0}\subset B(X).$ Its continuous counterpart can be found in \cite{Li-FPo-12}.

\begin{theorem}\label{ThFunctionalEquation}
Let $\{S_{\alpha,\beta}^n\}_{n\in\N_0}\subset B(X)$ be a discrete $(\alpha,\beta)$-resolvent family generated by $A.$ Then, the following functional equation holds
\begin{equation}\label{FunctionalEqn}
S_{\alpha,\beta}^m(k^\alpha_\tau\star S_{\alpha,\beta})^n-(k^\alpha_\tau\star S_{\alpha,\beta})^mS_{\alpha,\beta}^n=k^\beta_\tau(m)(k^\alpha_\tau\star S_{\alpha,\beta})^n-k^\beta_\tau(n)(k^\alpha_\tau\star S_{\alpha,\beta})^m,
\end{equation}
for all $m,n\in\N_0.$
\end{theorem}
\begin{proof}
For each $x\in X$ and $n\in \N_0$ we recall that
\begin{equation*}
S_{\alpha,\beta}^nx=k^\beta_\tau(n)x+\tau A\sum_{j=0}^n k^\alpha_\tau(n-j)S_{\alpha,\beta}^jx.
\end{equation*}
Let $n,m\in \mathbb{N}_0$. Then, from Definition \ref{DefResolvent} part (1) we have:
\begin{align*}
\left(  \sum_{j=0}^{m}k^\beta_\tau(m-j)S_{\alpha,\beta}^j\right)S_{\alpha,\beta}^nx&=\left(  \sum_{j=0}^{m}k^\beta_\tau(m-j)S_{\alpha,\beta}^j\right)\left[k^\beta_\tau(n)x+\tau A\sum_{j=0}^n k^\alpha_\tau(n-j)S_{\alpha,\beta}^jx.\right]\\
&= k^\beta_\tau(n)\sum_{j=0}^{m}k^\beta_\tau(m-j)S_{\alpha,\beta}^jx+ \sum_{j=0}^{m}k^\beta_\tau(m-j)S_{\alpha,\beta}^j\left[\tau A\sum_{j=0}^n k^\alpha_\tau(n-j)S_{\alpha,\beta}^jx\right]\\
&= k^\beta_\tau(n)\sum_{j=0}^{m}k^\beta_\tau(m-j)S_{\alpha,\beta}^jx+ \sum_{j=0}^{m}k^\beta_\tau(m-j)\tau AS_{\alpha,\beta}^j\left[\sum_{j=0}^n k^\alpha_\tau(n-j)S_{\alpha,\beta}^jx\right]\\
&= k^\beta_\tau(n)\sum_{j=0}^{m}k^\beta_\tau(m-j)S_{\alpha,\beta}^jx+ \tau A\sum_{j=0}^{m}k^\beta_\tau(m-j)S_{\alpha,\beta}^j\left[\sum_{j=0}^n k^\alpha_\tau(n-j)S_{\alpha,\beta}^jx\right]\\
&= k^\beta_\tau(n)\sum_{j=0}^{m}k^\beta_\tau(m-j)S_{\alpha,\beta}^jx+ \left(S_{\alpha,\beta}^m-k^\beta_\tau(m)\right)\left[\sum_{j=0}^n k^\alpha_\tau(n-j)S_{\alpha,\beta}^jx\right]\\
&= k^\beta_\tau(n)\sum_{j=0}^{m}k^\beta_\tau(m-j)S_{\alpha,\beta}^jx+ S_{\alpha,\beta}^m\sum_{j=0}^n k^\alpha_\tau(n-j)S_{\alpha,\beta}^jx\\
&\quad -k^\beta_\tau(m)\sum_{j=0}^n k^\alpha_\tau(n-j)S_{\alpha,\beta}^jx.
\end{align*}

Therefore, we obtain
\begin{align*}
\left(  \sum_{j=0}^{m}k^\beta_\tau(m-j)S_{\alpha,\beta}^j\right)S_{\alpha,\beta}^nx&= k^\beta_\tau(n)\sum_{j=0}^{m}k^\beta_\tau(m-j)S_{\alpha,\beta}^jx+ S_{\alpha,\beta}^m\sum_{j=0}^n k^\alpha_\tau(n-j)S_{\alpha,\beta}^jx\\
&\quad -k^\beta_\tau(m)\sum_{j=0}^n k^\alpha_\tau(n-j)S_{\alpha,\beta}^jx.
\end{align*}

Reorganizing the last equality we get the desired result and the proof is finished.
\end{proof}

\begin{theorem}\label{ThRepresentation}
Let $1<\alpha<2$ and $\beta\geq 1$ such that $\alpha-\beta+1>0.$ Assume that $A$ is $\omega$-sectorial of angle
$\frac{(\alpha-1)\pi}{2},$ where $\omega<0.$ Then $A$ generates an $(\alpha,\beta)$-resolvent sequence $\{S_{\alpha,\beta}^n\}_{n\in\N_0}\subset \mathcal{B}(X).$
\end{theorem}
\begin{proof}
By \cite[Theorem 2.5]{Pon-20a}, $A$ generates an exponentially bounded $(\alpha,\beta)$-resolvent family $\{S_{\alpha,\beta}(t)\}_{t\geq 0}$ such that $S_{\alpha,\beta}(t)Ax=AS_{\alpha,\beta}(t)x$ for all $x\in D(A)$ and $t\geq 0,$ and
\begin{equation}\label{eq3.1b}
S_{\alpha,\beta}(t)x=g_\beta(t)x+A\int_0^t g_\alpha(t-s)S_{\alpha,\beta}(s)xds,
\end{equation}
for all $x\in X$ and $t\geq 0,$ where for $\mu>0,$ $g_\mu(t):=t^{\mu-1}/\Gamma(\mu).$ For each $x\in X,$ define $S_{\alpha,\beta}^nx$ by
\begin{equation*}
S_{\alpha,\beta}^nx:=\int_0^\infty \rho_n^\tau(t)S_{\alpha,\beta}(t)xdt,\quad n\in\N_0.
\end{equation*}
Multiplying \eqref{eq3.1b} by $\rho^\tau_n(t)$ and integrating over $[0,\infty)$ we conclude by \cite[Theorem 5.2]{Ab-Li-Mi-Ve-19} or \cite[Theorem 2.8]{Pon-19} that
\begin{equation*}
S_{\alpha,\beta}^nx=k_\tau^\beta(n)x+A\int_0^\infty \rho_n^\tau(t)(g_\alpha\ast S_{\alpha,\beta})(t)xdt=k_\tau^\beta(n)x+\tau A\sum_{j=0}^n k_\tau^\alpha(n-j)S_{\alpha,\beta}^jx.
\end{equation*}
Finally, multiplying the identity $S_{\alpha,\beta}(t)Ax=AS_{\alpha,\beta}(t)x$ by $\rho^\tau_n(t)$ and integrating over $[0,\infty),$ we get $S_{\alpha,\beta}^nAx=AS_{\alpha,\beta}^nx$ for all $n\in\N_0$ and $x\in D(A).$
\end{proof}

\begin{theorem}\label{ThGen1}
Let $0<\alpha<1.$ Assume that $A$ is the generator of a $C_0$-semigroup $\{T(t)\}_{t\geq 0}.$ Then, $A$ generates the $(\alpha,1)$-resolvent sequence $\{S_{\alpha,1}^n\}_{n\in\N_0}$ given by
\begin{equation*}
S_{\alpha,1}^n x=\int_0^\infty \int_0^\infty \rho^\tau_n(t) \psi_{\alpha,1-\alpha}(t,s)T(s)xdsdt, \quad x\in X,
\end{equation*}
where $\psi_{\alpha,1-\alpha}$ is the Wright type function given by
\begin{eqnarray}\label{eqSubordination-4}
\hspace{-1.2cm}\notag\psi_{\alpha,1-\alpha}(t,s)=\frac{1}{\pi}\int_0^\infty\hspace{-0.8cm}
&&\rho^{\alpha-1}e^{-s\rho^{\alpha}\cos\alpha(\pi-\theta)-t\rho\cos\theta}\\
&&\times\sin\left(t\rho\sin\theta-s\rho^{\alpha}\sin\alpha(\pi-\theta)+\alpha(\pi-\theta)\right)d\rho,
\end{eqnarray}
for $\theta\in(\pi-\frac{\pi}{2\alpha},\pi/2).$
\end{theorem}
\begin{proof}
By \cite{Ba-01} or \cite[Corollary 2]{Pon-20c}, $A$ generates the fractional resolvent family $\{S_{\alpha,1}(t)\}_{t\geq 0}$ defined by
\begin{equation}\label{ThGen1-Eq1}
S_{\alpha,1}(t) x=\int_0^\infty \psi_{\alpha,1-\alpha}(t,s)T(s)xds, \quad x\in X,
\end{equation}
where $\psi_{\alpha,1-\alpha}(t,s)$ is defined in \eqref{eqSubordination-4}. For each $n\in\N_0,$ define $S_{\alpha,1}^n $ by
\begin{equation*}
S_{\alpha,1}^n:=\int_0^\infty \rho_n^\tau (t)S_{\alpha,1}(t)dt.
 \end{equation*}
Multiplying both sides in equation \eqref{ThGen1-Eq1} by $\rho^\tau_n(t)$ and integrating over $[0,\infty)$ we obtain the desired result.
\end{proof}

\begin{theorem}\label{ThGen2}
Let $0<\alpha<1.$ Assume that $A$ is the generator of a $C_0$-semigroup $\{T(t)\}_{t\geq 0}.$ Then, $A$ generates the $(\alpha,\alpha)$-resolvent sequence $\{S_{\alpha,\alpha}^n\}_{n\in\N_0}$ given by
\begin{equation*}
S_{\alpha,\alpha}^n x=\int_0^\infty \int_0^\infty \rho^\tau_n(t) \psi_{\alpha,0}(t,s)T(s)xdsdt, \quad x\in X,
\end{equation*}
where $\psi_{\alpha,0}$ is the Wright type function given by
\begin{equation}\label{eqSubordination-6}
\psi_{\alpha,0}(t,s)=\frac{1}{\pi}\int_0^\infty e^{t\rho\cos\theta-s\rho^\alpha\cos\alpha\theta}\cdot\sin(t\rho\sin\theta-s\rho\sin\alpha\theta+\theta)d\rho,
\end{equation}
for $\pi/2<\theta<\pi.$
\end{theorem}
\begin{proof}
By \cite[Theorem 3.1]{Ke-Li-Wa-13b} or \cite[Corollary 3]{Pon-20c}, $A$ generates the fractional resolvent family $\{S_{\alpha,\alpha}(t)\}_{t\geq 0}$ which is defined by
\begin{equation*}
S_{\alpha,\alpha}(t)x=\int_0^\infty\psi_{\alpha,0}(t,s)T(s)xds, \quad x\in X,
\end{equation*}
where $\psi_{\alpha,0}(t,s)$ is given in \eqref{eqSubordination-6}. Multiplying both sides in the last equation by $\rho^\tau_n(t)$ and integrating over $[0,\infty)$ the result follows as in the proof of Theorem \ref{ThGen1}.

\end{proof}

\begin{theorem}\label{ThGen3}
Let $1<\alpha<2.$ Assume that $A$ is the generator of a cosine family $\{C(t)\}_{t\in \R}.$ Then, $A$ generates the $(\alpha,1)$-resolvent sequence $\{S_{\alpha,1}^n\}_{n\in\N_0}$ given by
\begin{equation*}
S_{\alpha,1}^n x=\int_0^\infty \int_0^\infty \rho^\tau_n(t) \psi_{\frac{\alpha}{2},1-\frac{\alpha}{2}}(t,s)C(s)xdsdt, \quad x\in X,
\end{equation*}
where $\psi_{\frac{\alpha}{2},1-\frac{\alpha}{2}}$ is the Wright type function given by
\begin{eqnarray}\label{eqSubordination-8}
\hspace{-1.0cm}\notag\psi_{\frac{\alpha}{2},1-\frac{\alpha}{2}}(t,s)=\frac{1}{\pi}\int_0^\infty \hspace{-0.8cm}&&\rho^{\frac{\alpha}{2}-1}e^{-s\rho^{\frac{\alpha}{2}}\cos\frac{\alpha}{2}(\pi-\theta)-t\rho\cos\theta}\\
&&\times\sin\left(t\rho\sin\theta-s\rho^{\frac{\alpha}{2}}\sin\tfrac{\alpha}{2}(\pi-\theta)+\tfrac{\alpha}{2}(\pi-\theta)\right)d\rho,
\end{eqnarray}
\vskip -3pt \noindent
for $\theta\in(\pi-\frac{2}{\alpha},\pi/2).$
\end{theorem}
\begin{proof}
By \cite[Corollary 4]{Pon-20c}, $A$ generates the fractional resolvent family $\{S_{\alpha,1}(t)\}_{t\geq 0}$ given by \begin{equation*}
S_{\alpha,1}(t) x=\int_0^\infty \psi_{\frac{\alpha}{2},1-\frac{\alpha}{2}}(t,s)C(s)xds, \quad x\in X,
\end{equation*}
where $\psi_{\frac{\alpha}{2},1-\frac{\alpha}{2}}(t,s)$ is defined in \eqref{eqSubordination-8}. The result follows as in the Proof of Theorem \ref{ThGen2}.
\end{proof}

\begin{theorem}
Let $1<\alpha<2.$ Assume that $A$ is the generator of a cosine family $\{C(t)\}_{t\in \R}.$ Then, $A$ generates the $(\alpha,\alpha)$-resolvent sequence $\{S_{\alpha,\alpha}^n\}_{n\in\N_0}$ given by
\begin{equation*}
S_{\alpha,\alpha}^n x=\int_0^\infty \int_0^\infty \rho^\tau_n(t) \psi_{\frac{\alpha}{2},\frac{\alpha}{2}}(t,s)C(s)xdsdt, \quad x\in X,
\end{equation*}
where $\psi_{\frac{\alpha}{2},\frac{\alpha}{2}}$ is the Wright type function given by
\vskip -9pt
\begin{equation}\label{eqSubordination-10}
\psi_{\frac{\alpha}{2},\frac{\alpha}{2}}(t,s)=(g_{\frac{\alpha}{2}}\ast \psi_{\frac{\alpha}{2},0}(\cdot,s))(t),
\end{equation}
\vskip -2pt \noindent
where $\psi_{\frac{\alpha}{2},0}(\cdot,s)$ is given in \eqref{eqSubordination-6}.
\end{theorem}
\begin{proof}
By \cite[Corollary 5]{Pon-20c}, $A$ generates the fractional resolvent family $\{S_{\alpha,1}(t)\}_{t\geq 0}$ given by \begin{equation*}
S_{\alpha,\alpha}(t)x=\int_0^\infty \psi_{\frac{\alpha}{2},\frac{\alpha}{2}}(t,s)C(s)xds,\quad x\in X,
\end{equation*}
where $\psi_{\frac{\alpha}{2},\frac{\alpha}{2}}(t,s)$ is defined in \eqref{eqSubordination-10}. The rest of the proof follows as in Theorem \ref{ThGen1}.
\end{proof}

\begin{proposition}\label{PropUniqueness}
If $\{S_{\alpha,\beta}^n\}_{n\in\N_0}$ and $\{T_{\alpha,\beta}^n\}_{n\in\N_0}$ are $(\alpha,\beta)$-resolvent sequences generated by $A,$ then $S_{\alpha,\beta}^n=T_{\alpha,\beta}^n$ for all $n\in\N_0.$
\end{proposition}
\begin{proof}
For $x\in X,$ we define $h(n):=S_{\alpha,\beta}^nx-T_{\alpha,\beta}^nx.$ By Proposition \ref{Prop2.1}, we obtain
\begin{equation*}
S_{\alpha,\beta}^0x=T_{\alpha,\beta}^0x=k_\tau^\beta(0)\tau^{-\alpha}(\tau^{-\alpha}-A)^{-1},
\end{equation*}
which implies that $h(0)=0.$ On the other hand, by Definition \ref{DefResolvent}
\begin{equation*}
h(n)=\tau A\sum_{j=0}^n k_\tau^\alpha(n-j)h(j)
\end{equation*}
and thus
\begin{equation*}
(I-\tau^\alpha A)h(n)=\tau A\sum_{j=0}^{n-1}k_\tau^\alpha(n-j)h(j).
\end{equation*}
By Proposition \ref{Prop2.1}, $\tau^{-\alpha}\in\rho(A),$ and therefore $(I-\tau^\alpha A)=\tau^{\alpha}(\tau^{-\alpha}-A)$ is an invertible operator.
Hence,
\begin{equation*}
h(n)=0
\end{equation*}
for all $n\in \N.$ This implies that $S_{\alpha,\beta}^nx=T_{\alpha,\beta}^nx$ for all $n\in \N_0$ and $x\in X.$
\end{proof}

Now, we define the following sequence $(a_{n,l})$ as:

\begin{equation*}
a_{0,1}:=k_\tau^\beta(0),\quad a_{1,1}:=(k_\tau^\beta(1)k_\tau^\alpha(1)-k_\tau^\beta(0)k_\tau^\alpha(1))k_\tau^{\alpha}(0)^{-1},\quad a_{1,2}:=k_\tau^\beta(0)k_\tau^\alpha(1)k_\tau^{\alpha}(0)^{-1}
\end{equation*}
and for $n\geq 2,$ we define $(a_{n,l})$ as follow:
\begin{equation*}\label{DefSeqA3}
a_{n,n+1}:=k_\tau^\alpha(1)a_{n-1,n}k_\tau^{\alpha}(0)^{-1}.
\end{equation*}

        \begin{equation}\label{DefSeqA2}
            a_{n,1}:=\left(k_\tau^\beta(n)k_\tau^\alpha(0)-\sum_{j=0}^{n-1}k_\tau^\alpha(n-j)a_{j,1}\right)k_\tau^{\alpha}(0)^{-1},
        \end{equation}

\begin{equation}\label{DefSeqA1}
            a_{n,l}:=\left(\sum_{j=l-2}^{n-1}k_\tau^\alpha(n-j)a_{j,l-1}-\sum_{j=l-1}^{n-1}k_\tau^\alpha(n-j)a_{j,l}\right)k_\tau^{\alpha}(0)^{-1}, \quad \mbox{ for } 2\leq l\leq n,
\end{equation}

Moreover, we denote the resolvent operator $R_\tau:X\to D(A)$ as
\begin{equation*}
R_\tau:=\tau^{-\alpha}\left(\tau^{-\alpha}-A\right)^{-1}.
\end{equation*}

The next Theorem is one of the main result in this paper and gives an explicit representation of the discrete resolvent family $S_{\alpha,\beta}^n$ for all $n\in \N_0.$

\begin{theorem}\label{ThRepresentation}
Let $\{S_{\alpha,\beta}^n\}_{n\in\N_0}\subset B(X)$ be a discrete $(\alpha,\beta)$-resolvent family generated by $A.$ Then, for each $x\in X,$
\begin{equation}\label{EqResolv1}
S_{\alpha,\beta}^0x=a_{0,1}R_\tau x,\quad \mbox{ and } \quad S_{\alpha,\beta}^1=a_{1,1}R_\tau x+a_{1,2}R_\tau^2 x,
\end{equation}
and for $n\geq 2$
\begin{equation}\label{EqResolv2}
  S_{\alpha,\beta}^nx=\sum_{j=1}^{n+1}a_{n,j}R_\tau^j x.
\end{equation}
\end{theorem}
\begin{proof}
The first identity in \eqref{EqResolv1} follows from \eqref{Prop2.1}. In order to prove the second one, we take $m=1, n=0$ in  \eqref{FunctionalEqn} and we get
\begin{eqnarray*}
S_{\alpha,\beta}^1 k_\tau^\alpha(0)S_{\alpha,\beta}^0 x-\left(\sum_{j=0}^1 k_\tau^\alpha(1-j)S_{\alpha,\beta}^j\right)S_{\alpha,\beta}^0 x=k_\tau^\beta(1)k_\tau^\alpha(0)S_{\alpha,\beta}^0x-k_\tau^\beta(0)\left(\sum_{j=0}^1 k_\tau^\alpha(1-j)S_{\alpha,\beta}^j x\right),
\end{eqnarray*}
which is equivalent to
\begin{equation*}
-k_\tau^\alpha(1)S_{\alpha,\beta}^0S_{\alpha,\beta}^0x=k_\tau^\beta(1)k_\tau^\alpha(0)S_{\alpha,\beta}^0x-k_\tau^\beta(0)k_\tau^\alpha(1)S_{\alpha,\beta}^0x-k_\tau^\beta(0)k_\tau^\alpha(0)S_{\alpha,\beta}^1x. \end{equation*}
As $S_{\alpha,\beta}^0x=k_\tau^\beta(0)R_\tau x$ this last identity implies that
\begin{eqnarray*}
k_\tau^\beta(0)k_\tau^\alpha(0)S_{\alpha,\beta}^1x=k_\tau^\beta(1)k_\tau^\alpha(0)k_\tau^\beta(0)R_\tau x-k_\tau^\beta(0)k_\tau^\alpha(1)k_\tau^\beta(0)R_\tau x+k_\tau^\beta(0)k_\tau^\alpha(1)k_\tau^\beta(0)R_\tau^2x.
\end{eqnarray*}
Since $k_\tau^\alpha(0)=\tau^{\alpha-1}\neq 0,$ we conclude that
\begin{equation*}
S_{\alpha,\beta}^1x=(k_\tau^\beta(1)k_\tau^\alpha(0)-k_\tau^\beta(0)k_\tau^\alpha(1))k_\tau^{\alpha}(0)^{-1}R_\tau+k_\tau^\beta(0)k_\tau^\alpha(1)k_\tau^{\alpha}(0)^{-1} R_\tau^2x=a_{1,1}R_\tau x+a_{1,2}R_\tau^2x.
\end{equation*}
In order to prove \eqref{EqResolv2} we proceed by induction on $n\geq 2.$ For $n=2,$ we take $m=2$ and $n=0$ in  \eqref{FunctionalEqn} to obtain
\begin{eqnarray*}
S_{\alpha,\beta}^2 k_\tau^\alpha(0)S_{\alpha,\beta}^0 x-\left(\sum_{j=0}^2 k_\tau^\alpha(2-j)S_{\alpha,\beta}^j\right)S_{\alpha,\beta}^0 x=k_\tau^\beta(2)k_\tau^\alpha(0)S_{\alpha,\beta}^0x-k_\tau^\beta(0)\left(\sum_{j=0}^2 k_\tau^\alpha(2-j)S_{\alpha,\beta}^j x\right),
\end{eqnarray*}
which can be written as
\begin{eqnarray*}
-k_\tau^\alpha(2)S_{\alpha,\beta}^0S_{\alpha,\beta}^0x-k_\tau^\alpha(1)S_{\alpha,\beta}^1S_{\alpha,\beta}^0x&=&k_\tau^\beta(2)k_\tau^\alpha(0)S_{\alpha,\beta}^0x-k_\tau^\beta(0)k_\tau^\alpha(2)S_{\alpha,\beta}^0x-k_\tau^\beta(0)k_\tau^\alpha(1)S_{\alpha,\beta}^1x\\
&&-k_\tau^\beta(0)k_\tau^\alpha(0)S_{\alpha,\beta}^2x.
\end{eqnarray*}
Since $S_{\alpha,\beta}^0x=k_\tau^\beta(0)R_\tau x$ and $S_{\alpha,\beta}^1=a_{1,1}R_\tau x+a_{1,2}R_\tau^2 x$ we have
\begin{eqnarray*}
k_\tau^\beta(0)k_\tau^\alpha(0)S_{\alpha,\beta}^2x&=&k_\tau^\beta(2)k_\tau^\alpha(0)k_\tau^\beta(0)R_\tau x-k_\tau^\beta(0)k_\tau^\alpha(2)k_\tau^\beta(0)R_\tau x-k_\tau^\beta(0)k_\tau^\alpha(1)a_{1,1}R_\tau x\\
&&-k_\tau^\beta(0)k_\tau^\alpha(1)a_{1,2}R_\tau^2x+k_\tau^\alpha(2)k_\tau^\beta(0)k_\tau^\beta(0)R_\tau^2x\\
&&+k_\tau^\alpha(1)k_\tau^\beta(0)a_{1,1}R_\tau^2x+k_\tau^\alpha(1)k_\tau^\beta(0)a_{1,2}R_\tau^3x.
\end{eqnarray*}
Hence,
\begin{eqnarray}\label{eq2.3}
\nonumber k_\tau^\alpha(0)S_{\alpha,\beta}^2x&=&(k_\tau^\beta(2)k_\tau^\alpha(0)-k_\tau^\beta(0)k_\tau^\alpha(2)-k_\tau^\alpha(1)a_{1,1})R_\tau x\\
&&+(k_\tau^\alpha(2)k_\tau^\beta(0)+k_\tau^\alpha(1)a_{1,1}-k_\tau^\alpha(1)a_{1,2})R_\tau^2 x+k_\tau^\alpha(1)a_{1,2}R_\tau^3x.
\end{eqnarray}
On the other hand, if we expand the sum (for $n=2$) in \eqref{EqResolv2} and we obtain
\begin{eqnarray*}
\sum_{j=1}^{3}a_{n,j}R_\tau^j x=a_{2,1}R_\tau x+a_{2,2}R_\tau^2x+a_{2,3}R_\tau^3x,
\end{eqnarray*}
and by definition of the sequence $(a_{n,l})$ we get
\begin{eqnarray*}
a_{2,1}&=&(k_\tau^\beta(2)k_\tau^\alpha(0)-k_\tau^\alpha(2)a_{0,1}-k_\tau^\alpha(1)a_{1,1})k_\tau^{\alpha}(0)^{-1}\\
a_{2,2}&=&(k_\tau^\alpha(2)a_{0,1}+k_\tau^\alpha(1)a_{1,1}-k_\tau^\alpha(1)a_{1,2})k_\tau^{\alpha}(0)^{-1}\\
a_{2,3}&=&k_\tau^\alpha(1)a_{1,2}k_\tau^{\alpha}(0)^{-1}.
\end{eqnarray*}
From \eqref{eq2.3} we conclude that
\begin{eqnarray*}
S_{\alpha,\beta}^2x=\sum_{j=1}^{3}a_{n,j}R_\tau^j x.
\end{eqnarray*}
Now, we assume that \eqref{EqResolv2} holds for all $l\leq n.$ In order to prove the identity for $n+1,$ we first take $m=n+1$ and $n=0$ in \eqref{FunctionalEqn} to obtain
\begin{eqnarray*}
S_{\alpha,\beta}^{n+1} k_\tau^\alpha(0)S_{\alpha,\beta}^0 x-\left(\sum_{j=0}^{n+1} k_\tau^\alpha(n+1-j)S_{\alpha,\beta}^j\right)S_{\alpha,\beta}^0 x&=&k_\tau^\beta(n+1)k_\tau^\alpha(0)S_{\alpha,\beta}^0x\\
&&-k_\tau^\beta(0)\left(\sum_{j=0}^{n+1} k_\tau^\alpha(n+1-j)S_{\alpha,\beta}^j x\right).
\end{eqnarray*}
Hence,
\begin{eqnarray*}
S_{\alpha,\beta}^{n+1} k_\tau^\alpha(0)S_{\alpha,\beta}^0 x-k_\tau^\alpha(n+1)S_{\alpha,\beta}^0S_{\alpha,\beta}^0x-k_\tau^\alpha(n)S_{\alpha,\beta}^1S_{\alpha,\beta}^0x-...-k_\tau^\alpha(0)S_{\alpha,\beta}^{n+1}S_{\alpha,\beta}^0x&=&\\
&&\hspace{-11.5cm}k_\tau^\beta(n+1)k_\tau^\alpha(0)S_{\alpha,\beta}^0x-k_\tau^\beta(0)k_\tau^\alpha(n+1)S_{\alpha,\beta}^0x-k_\tau^\beta(0)k_\tau^\alpha(n)S_{\alpha,\beta}^1x-...-k_\tau^\beta(0)k_\tau^\alpha(0)S_{\alpha,\beta}^{n+1}x.
\end{eqnarray*}
That is,
\begin{eqnarray*}
k_\tau^\beta(0)k_\tau^\alpha(0)S_{\alpha,\beta}^{n+1}x&=&k_\tau^\beta(n+1)k_\tau^\alpha(0)S_{\alpha,\beta}^0x-k_\tau^\beta(0)k_\tau^\alpha(n+1)S_{\alpha,\beta}^0x-...-k_\tau^\beta(0)k_\tau^\alpha(1)S_{\alpha,\beta}^nx\\
&&+k_\tau^\alpha(n+1)S_{\alpha,\beta}^0S_{\alpha,\beta}^0x+k_\tau^\alpha(n)S_{\alpha,\beta}^1S_{\alpha,\beta}^0x+...+k_\tau^\alpha(1)S_{\alpha,\beta}^{n}S_{\alpha,\beta}^0x.
\end{eqnarray*}
%\begin{eqnarray*}
%k_\tau^\beta(0)k_\tau^\alpha(0)S_{\alpha,\beta}^{n+1}x&=&k_\tau^\beta(n+1)k_\tau^\alpha(0)S_{\alpha,\beta}^0x-k_\tau^\beta(0)k_\tau^\alpha(n+1)S_{\alpha,\beta}^0x-k_\tau^\beta(0)k_\tau^\alpha(n)S_{\alpha,\beta}^1x-...-k_\tau^\beta(0)k_\tau^\alpha(1)S_{\alpha,\beta}^nx\\
%&&+k_\tau^\alpha(n+1)S_{\alpha,\beta}^0S_{\alpha,\beta}^0x+k_\tau^\alpha(n)S_{\alpha,\beta}^1S_{\alpha,\beta}^0x+...+k_\tau^\alpha(1)S_{\alpha,\beta}^{n}S_{\alpha,\beta}^0x.
%\end{eqnarray*}
Since $S_{\alpha,\beta}^0x=k_\tau^\beta(0)R_\tau x$ we can write this last identity as
\begin{eqnarray*}
k_\tau^\alpha(0)S_{\alpha,\beta}^{n+1}x&=&k_\tau^\beta(n+1)k_\tau^\alpha(0)R_\tau x-k_\tau^\alpha(n+1)S_{\alpha,\beta}^0x-k_\tau^\alpha(n)S_{\alpha,\beta}^1x-...-k_\tau^\alpha(1)S_{\alpha,\beta}^nx\\
&&+k_\tau^\alpha(n+1)R_\tau S_{\alpha,\beta}^0x+k_\tau^\alpha(n)S_{\alpha,\beta}^1R_\tau x+...+k_\tau^\alpha(1)S_{\alpha,\beta}^{n}R_\tau x.
\end{eqnarray*}
By induction hypothesis we have
\begin{eqnarray*}
k_\tau^\alpha(0)S_{\alpha,\beta}^{n+1}x&=&k_\tau^\beta(n+1)k_\tau^\alpha(0)R_\tau x-k_\tau^\alpha(n+1)k_\tau^\beta(0)R_\tau x-k_\tau^\alpha(n)a_{1,1}R_\tau x-k_\tau^\alpha(n)a_{1,2}R_\tau^2 x-...\\
&&...-k_\tau^\alpha(1)[a_{n,1}R_\tau+...+a_{n,n}R_{\tau}^n+a_{n,n+1}R_\tau^{n+1}]x\\
&&...+k_\tau^\alpha(n+1)k_\tau^\beta(0)R_\tau^2x+k_\tau^\alpha(n)R_\tau [a_{1,1}R_\tau+a_{1,2}R_\tau^2]x+...\\
&&+k_\tau^\alpha(1)R_\tau [a_{n,1}R_\tau+...+a_{n,n}R_{\tau}^n+a_{n,n+1}R_\tau^{n+1}]x\\
&=&\left(k_\tau^\beta(n+1)k_\tau^\alpha(0)-k_\tau^\alpha(n+1)k_\tau^\beta(0)-k_\tau^\alpha(n)a_{1,1}-...-k_\tau^\alpha(1)a_{n,1}\right)R_\tau x\\
&&+\left(k_\tau^\alpha(n+1)k_\tau^\beta(0)+k_\tau^\alpha(n)a_{1,1}+...+k_\tau^\alpha(1)a_{n,1}-k_\tau^\alpha(n)a_{1,2}-...-k_\tau^\alpha(1)a_{n,2}\right)R_\tau^2x\\
&&\vdots\\
&&+\left(k_\tau^\alpha(2)a_{n-1,n}+k_\tau^\alpha(1)a_{n,n}-k_\tau^\alpha(1)a_{n,n+1}\right)R_{\tau}^{n+1}x\\
&&+k_\tau^\alpha(1)a_{n,n+1}R_\tau^{n+2}x,
\end{eqnarray*}
and therefore
\begin{equation*}
S_{\alpha,\beta}^{n+1}x=a_{n+1,1}R_\tau x+a_{n+1,2}R_\tau^2x+...+a_{n+1,n+2}R_\tau^{n+2}x.
\end{equation*}
This finishes the proof.

\end{proof}

If $A$ is a bounded operator, we have the following result.

\begin{proposition}\label{Th-Representation}
Let $\alpha,\beta>0$ such that $\tau^\alpha<1.$ If $A$ is a bounded operator with $\|A\|<1,$ then $A$ generates the $(\alpha,\beta)$-resolvent sequence $\{S_{\alpha,\beta}^n\}_{n\in\N_0}$ defined by
\begin{equation}\label{Eq-Representation}
S_{\alpha,\beta}^n=\sum_{j=0}^\infty k_\tau^{\alpha j+\beta}(n)A^j.
\end{equation}
\end{proposition}
\begin{proof}
Let $x\in X$ and $n\in \N_0.$ From \cite[Formula 8.328]{Gr-Ry-00} the serie in \eqref{Eq-Representation} converges for $\tau^\alpha<1$ and $\|A\|<1.$ Then, by \eqref{Semigroup-for-K} we get
\begin{eqnarray*}
\tau A\sum_{j=0}^n k_\tau^\alpha(n-j)S_{\alpha,\beta}^jx&=& \sum_{l=0}^\infty A^{l+1}\tau \sum_{j=0}^n k_\tau^\alpha(n-j)k_\tau^{\alpha l+\beta}(j)x\\
&=&\sum_{l=0}^\infty A^{l+1}k_\tau^{\alpha(l+1)+\beta}(n)x\\
&=&\sum_{j=0}^\infty A^{j}k_\tau^{\alpha j+\beta}(n)x-k_\tau^\beta(n)x.
\end{eqnarray*}
Hence,
\begin{equation*}
k_\tau^\beta(n)x+\tau A\sum_{j=0}^n k_\tau^\alpha(n-j)S_{\alpha,\beta}^jx=\sum_{j=0}^\infty A^{j}k_\tau^{\alpha j+\beta}(n)x,
\end{equation*}
that is,
\begin{equation*}
S_{\alpha,\beta}^n=\sum_{j=0}^\infty k_\tau^{\alpha j+\beta}(n)A^j.
\end{equation*}
\end{proof}
The next Corollary is a direct consequence of Proposition \ref{Th-Representation}.
\begin{corollary}
Let $\tau<1.$ If $A$ is a bounded operator with $\|A\|<1,$ then $A$ generates the $(1,1)$-resolvent sequence $\{S_{1,1}^n\}_{n\in\N_0}$ defined by
\begin{equation}\label{SemigroupCase}
S_{1,1}^n=\sum_{j=0}^\infty k_\tau^{j+1}(n)A^j=\sum_{j=0}^\infty  \tau^{j}\frac{\Gamma(j+n+1)}{\Gamma(j+1)\Gamma(n+1)}A^j=\sum_{j=0}^\infty  \tau^{j}{n+j\choose j} A^j.
\end{equation}
\end{corollary}
Now, since for any $\beta>0$
\begin{equation*}
k_1^\beta(n)=\frac{n^{\beta-1}}{\Gamma(\beta)}\left(1+O\left(\frac{1}{n}\right)\right), n\in\N, \beta>0
\end{equation*}
(see \cite[Formula 8.328]{Gr-Ry-00}) we get
\begin{equation*}
\frac{\Gamma(j+n+1)}{\Gamma(j+1)\Gamma(n+1)}=k_1^{j+1}(n)=\frac{n^{j}}{j!}\left(1+O\left(\frac{1}{n}\right)\right)
\end{equation*}
and therefore, the identity \eqref{SemigroupCase} gives an approximation of the semigroup
\begin{equation*}
e^{tA}:=\sum_{j=0}^\infty  \frac{(tA)^j}{j!}
\end{equation*}
at $t_n:=n\tau,$ that is, $S_{1,1}^n$ approximates $e^{t_nA}$ for each $n\in\N_0.$

\begin{remark} For $n\in \mathbb{N}$ given, we define the matrix $A\in\mathbb{M}_{n+1}(\mathbb{R})$ and the vector $R\in \mathbb{R}^{n+1}$ as follow:
\begin{align*}
A(i,j):=\begin{cases}
a_{i-1,j},& i\geq j,\\
0, &j	>i.
\end{cases}
\quad R(i):=R^i, \quad i=1,...,n+1.
\end{align*}
Then, $S\in\mathbb{M}_{(n+1)\times 1}(\mathbb{R})$ defined by
\begin{align*}
S(i)=S_{\alpha,\beta}^{i-1}, \quad i=1,...,n+1.
\end{align*}
satisfies $S=AR^T$. Furthermore, it is not difficult to see that for the case $\alpha=\beta=1$, the matrix $A$ corresponds to unity $I_{n+1}$.

\end{remark}

We notice that if $e_{\alpha,\beta}(t):=t^{\beta-1}E_{\alpha,\beta}(-\varrho t^\alpha),$ where $\varrho>0,$ then $\{S_{\alpha,\beta}^n\}_{n\in\N_0},$ corresponds to a discretization of $e_{\alpha,\beta}(t)$ on the interval $[0,T].$ In Figure \ref{Figura2}, we illustrate the function $e_{\alpha,\beta}(t):=t^{\beta-1}E_{\alpha,\beta}(-\varrho t^\alpha)$ and the sequence $S_{\alpha,\beta}^n$ (generated by $A=\varrho I$) on the interval $[0,1],$ where $\tau=1/N,$ $0\leq n\leq N$ and $N=100,$ respectively. For $\varrho=1,$ we choose, respectively, $\alpha=1.1,\beta=0.1,$ and $\alpha=0.1,\beta=0.9.$

\begin{figure}[h]
\centering
\includegraphics[scale=0.24]{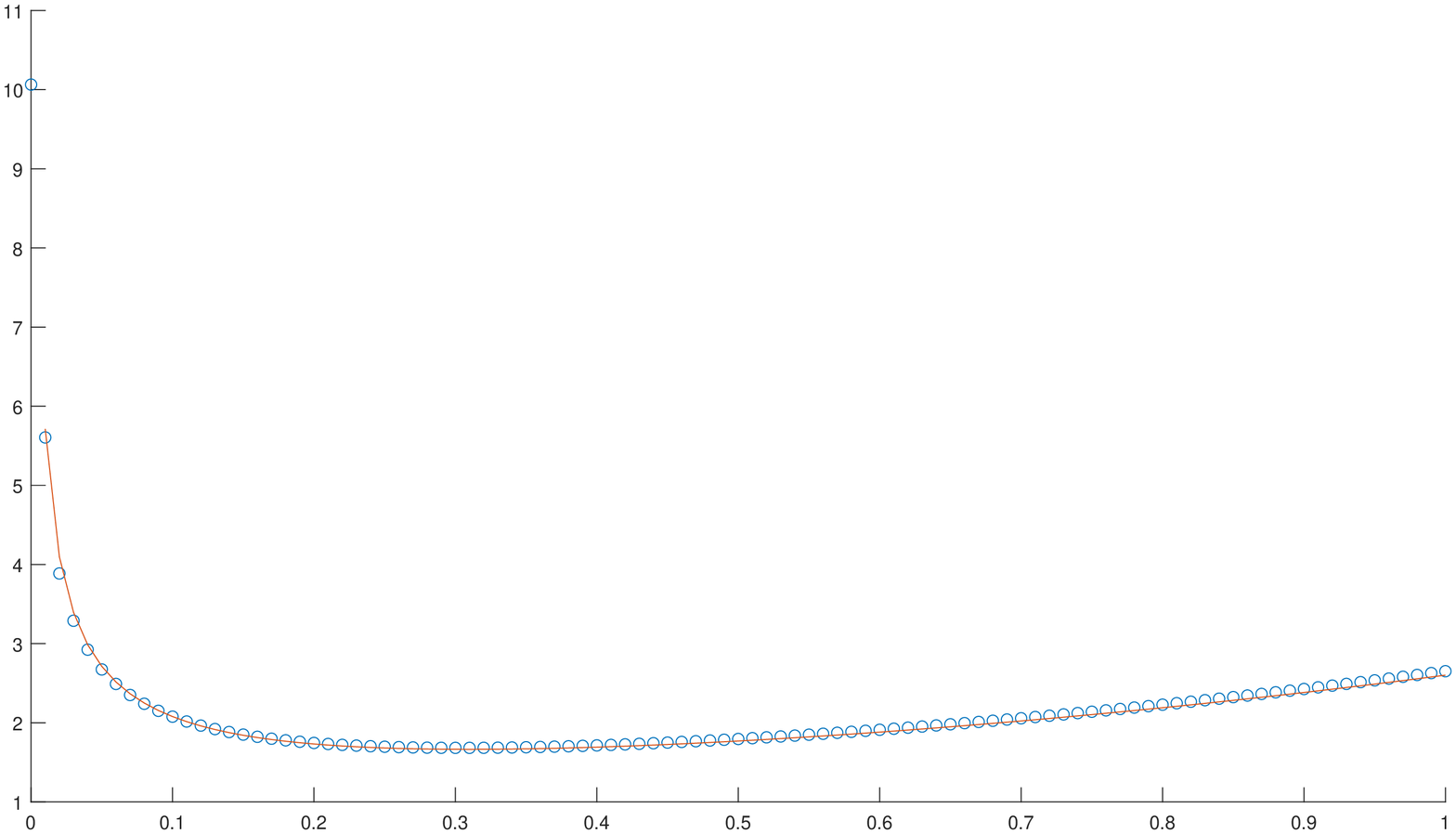}%%% OK %%%
\includegraphics[scale=0.24]{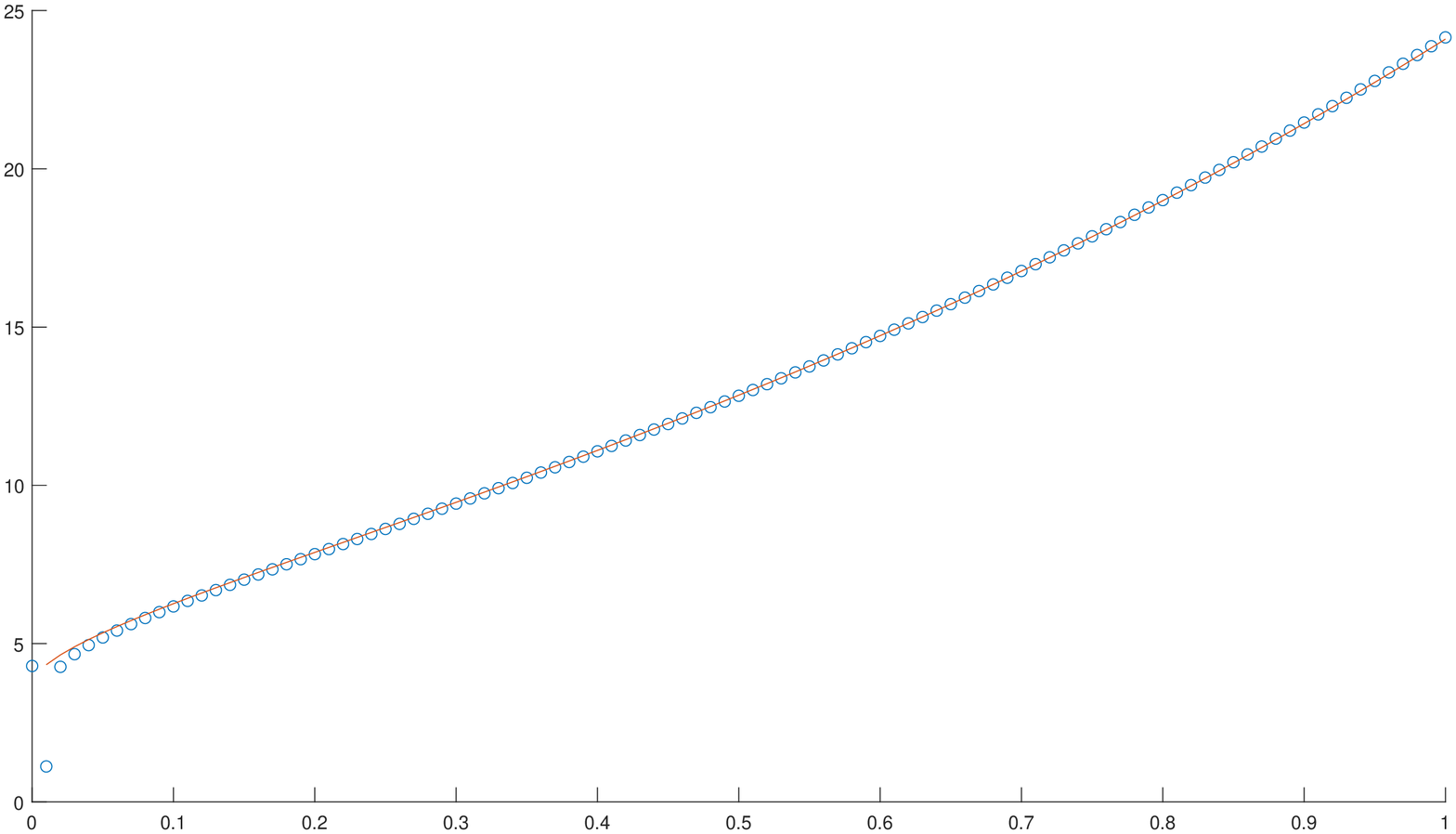}
\caption{$e_{\alpha,\beta}(t)$ (line) and $S_{\alpha,\beta}^n$ (circles) for $N=100.$}
\label{Figura2}
\end{figure}

\section{Solution to a fractional difference equations}\label{Sect4}

In this section, we study the existence and uniqueness of solutions to a fractional difference equation. The results in this section shows that the discrete resolvent families play a crucial role in the representation of solutions.

To illustrate the previous results, we consider the initial value problem
\begin{eqnarray}\label{MainEquation}
\left\{ \begin{array}{cll}
\,_C\Na^{\alpha} u^n&=&Au^n+\,_C\Na^{\alpha-1}f^n, \quad n\geq 2,\\
u^0&=&x_0, \\
u^1&=&0, \\
\end{array} \right.
\end{eqnarray}
where $1<\alpha<2,$ $A$ is a closed linear operator in a Banach space $X$ and $x_0\in X.$ We notice that \eqref{MainEquation} can be see as a discretization of the problem
\begin{equation}\label{Eq-Caputo-Continuous}
\partial_t^\alpha u(t)=Au(t)+\partial_t^{\alpha-1}f(t),\quad t>0,
\end{equation}
under the initial conditions $u(0)=x_0$ and $u'(0)=0,$ where $\partial_t^\alpha$ denotes the Caputo fractional derivative. This equation has been widely studied in the last years, see for instance \cite{Ar-Li-08,Cu-Li-08,Cu-So-09,Li-Ng-13} and references therein. By \cite{Cu-07} or \cite{Pon-20a}, if $A$ generates an exponentially bounded $(\alpha,1)$-resolvent family $\{S_{\alpha,1}(t)\}_{t\geq 0}$ in the sense of \eqref{eq3.1b}, then the solution to \eqref{Eq-Caputo-Continuous} is given by
\begin{equation*}
u(t)=S_{\alpha,1}(t)x_0+\int_0^t S_{\alpha,1}(t-s)f(s)ds.
\end{equation*}

The next result shows that the solution to \eqref{MainEquation} can be written as a discrete variation of parameter formula, similarly to the continuous case.

\begin{theorem}\label{Th-Caputo1-Resolvent}
Let $\tau>0$ and $1<\alpha<2.$ Let $A$ be the generator of an $(\alpha,1)$-discrete resolvent sequence $\{S_{\alpha,1}^n\}_{n\in \N_0}.$ If $x_0\in X,$ then the Caputo fractional difference equation \eqref{MainEquation} has a unique solution given by
\begin{equation*}\label{Sol-Resolvent}
u^n=S_{\alpha,1}^n x_0+\tau(S_{\alpha,1}\star f)^n,
\end{equation*}
for all $n\geq 2$ and $u^0=x_0,$ $u^1=0.$
\end{theorem}
\begin{proof}
Since $A$ generates an $(\alpha,1)$-discrete resolvent sequence $\{S_{\alpha,1}^n\}_{n\in\N_0}$ and $k^1_\tau(n)=1$ for all $n\in \N_0,$ we have, by definition, that
\begin{equation}\label{eqDiscreteSa1}
{S}_{\alpha,1}^jx=x+\tau A\sum_{l=0}^j k^{\alpha}_\tau(j-l){S}^l_{\alpha,1}x,
\end{equation}
for all $j\geq 0$ and $x\in X.$ By definition of the Caputo fractional backward difference operator for $1<\alpha<2,$ we have for all $n\geq 2$ that
\begin{equation}\label{eq3.1a}
_C\Na^\alpha (S_{\alpha,1}x)^n=\Na^{-(2-\alpha)}_\tau \Na^2_\tau (S_{\alpha,1}x)^n=\tau \sum_{j=0}^n k^{2-\alpha}_\tau(n-j)(\Na^2_\tau S_{\alpha,1}x)^j.
\end{equation}
The equality \eqref{eqDiscreteSa1} implies that
\begin{eqnarray*}
(\Na^2_\tau S_{\alpha,1}x)^j&=&\frac{1}{\tau^2}(S_{\alpha,1}^jx-2S_{\alpha,1}^{j-1}x+S_{\alpha,1}^{j-2}x)\\
&=&\frac{A}{\tau^2}\left[\tau \sum_{l=0}^j k^{\alpha}_\tau(j-l){S}^l_{\alpha,1}x-2\tau \sum_{l=0}^{j-1} k^{\alpha}_\tau(j-1-l){S}^l_{\alpha,1}x+\tau \sum_{l=0}^{j-2} k^{\alpha}_\tau(j-2-l){S}^l_{\alpha,1}x\right],
\end{eqnarray*}
for all $j\geq 2$ and $x\in X.$ Since $k^1_\tau(n)=1$ for all $n\in \N_0,$ the convolution property \eqref{Semigroup-for-K} implies that
\begin{eqnarray*}
\tau \sum_{j=0}^n k^{2-\alpha}_\tau(n-j)\tau \sum_{l=0}^j k^{\alpha}_\tau(j-l){S}^l_{\alpha,1}x&=&\tau^2\sum_{j=0}^n k^{2-\alpha}_\tau(n-j)(k_\tau^\alpha\star S_{\alpha,1})^j x\nonumber\\
&=&\tau^2(k_\tau^{2-\alpha}\star (k_\tau^\alpha\star S_{\alpha,1}))^n x\nonumber\\
&=&\tau(k_\tau^{2}\star S_{\alpha,1}))^n x\nonumber\\
&=&\tau^2(k_\tau^{1}\star (k_\tau^1\star S_{\alpha,1}))^n x\nonumber\\
&=&\tau^2\sum_{j=0}^n (k_\tau^{1}\star S_{\alpha,1})^j x \nonumber\\
&=&\tau^2\sum_{j=0}^n \sum_{l=0}^j S_{\alpha,1}^l x, \label{eq3.1aa}
\end{eqnarray*}
for all $n\geq 2.$ Since $\sum_{j=0}^{-k} v^j=0$ for all $k\in\N,$ we get similarly that
\begin{equation*}\label{eq3.3}
\tau \sum_{j=0}^n k^{2-\alpha}_\tau(n-j)\tau \sum_{l=0}^{j-1} k^{\alpha}_\tau(j-1-l){S}^l_{\alpha,1}x=\tau^2 \sum_{j=0}^{n-1} \sum_{l=0}^j{S}_{\alpha,1}^lx,
\end{equation*}
and \begin{equation}\label{eq3.4}
\tau \sum_{j=0}^n k^{2-\alpha}_\tau(n-j)\tau \sum_{l=0}^{j-2} k^{\alpha}_\tau(j-2-l){S}^l_{\alpha,1}x=\tau^2 \sum_{j=0}^{n-2} \sum_{l=0}^j{S}_{\alpha,1}^lx,
\end{equation}
for all $n\geq 2.$ By \eqref{eq3.1a}--\eqref{eq3.4} we obtain
\begin{eqnarray*}
_C\Na^\alpha (S_{\alpha,1}x)^n&=&A\left[\sum_{j=0}^n \sum_{l=0}^j{S}_{\alpha}^lx-2\sum_{j=0}^{n-1} \sum_{l=0}^j{S}_{\alpha,1}^lx+\sum_{j=0}^{n-2} \sum_{l=0}^j{S}_{\alpha,1}^lx\right]\\
&=&AS^n_{\alpha}x,
\end{eqnarray*}
for all $n\geq 2$ and $x\in X,$ and therefore
\begin{equation*}\label{EqA.1}
_C\Na^\alpha{S}^n_{\alpha,1}x_0=A{S}^n_{\alpha,1}x_0.
\end{equation*}
On the other hand, by definition we have
\begin{equation*}
_C\Na^\alpha (({S}_{\alpha,1}\star f)^n)=\Na^{-(2-\alpha)}_\tau \Na^2_\tau (({S}_{\alpha,1}\star f))^n=\tau \sum_{j=0}^n k^{2-\alpha}_\tau(n-j)\Na^2_\tau (\tau(S_{\alpha,1}\star f)^j),
\end{equation*}
for all $n\geq 2.$ Since
\begin{equation*}
\Na^2_\tau (S_{\alpha,1}\star f)^j=\frac{1}{\tau^2}\left[(S_{\alpha,1}\star f)^j-2(S_{\alpha,1}\star f)^{j-1}+(S_{\alpha,1}\star f)^{j-2}\right],
\end{equation*}
for all $j\geq 2,$ and by definition
\begin{equation*}
S_{\alpha,1}^n x=k_\tau^{1}(n)x+\tau A\sum_{j=0}^n k^{\alpha}_\tau(n-j)S_{\alpha,1}^jx=x+\tau A(k^\alpha_\tau\star S_{\alpha,1})^nx,
\end{equation*}
for all $x\in X,$ and $n\in \N_0,$ we get that
\begin{equation*}
(S_{\alpha,1} \star f)^n=\sum_{j=0}^n f^j+\tau A(k^\alpha_\tau\star S_{\alpha,1}\star f)^n,
\end{equation*}
for all $n\in\N_0.$ Hence
\begin{eqnarray*}
_C\Na^\alpha (( {S}_{\alpha,1}\star f)^n)&=&\tau \sum_{j=0}^n k^{2-\alpha}_\tau(n-j)\Na^2_\tau (S_{\alpha,1}\star f)^j\\
&=&\frac{1}{\tau^2}\sum_{j=0}^n k^{2-\alpha}_\tau(n-j)\left[\tau\sum_{l=0}^j f^l-2\tau\sum_{l=0}^{j-1} f^l+\tau\sum_{l=0}^{j-2} f^l\right]\\
&+&\frac{A}{\tau^2}\sum_{j=0}^n k^{2-\alpha}_\tau(n-j)\left[\tau (k^\alpha_\tau\star S_{\alpha,1}\star f)^j-2\tau (k^\alpha_\tau\star S_{\alpha,1}\star f)^{j-1}+\tau (k^\alpha_\tau\star S_{\alpha,1}\star f)^{j-2}\right],
\end{eqnarray*}
for all $n\geq 2.$

An easy computation shows that
\begin{equation}\label{eq2.4a}
\left[\tau\sum_{l=0}^j f^l-2\tau\sum_{l=0}^{j-1} f^l+\tau\sum_{l=0}^{j-2} f^l\right]=\tau^2 \frac{(f^j-f^{j-1})}{\tau}=\tau^2\Na^1_\tau(f)^j.
\end{equation}
Moreover, by \eqref{Semigroup-for-K}, we obtain
\begin{eqnarray*}
\tau \sum_{j=0}^n k^{2-\alpha}_\tau(n-j)(k^\alpha_\tau\star S_{\alpha,1}\star f)^j&=&\tau (k_\tau^{2-\alpha}\star k^\alpha\star S_{\alpha,\beta}\star f)^n\nonumber\\
&=&\tau^2(k_\tau^1\star k_\tau^1\star S_{\alpha,\beta}\star f)^n\nonumber\\
&=&\tau^2\sum_{j=0}^n k_\tau^1(n-j)(k_\tau^1\star S_{\alpha,\beta}\star f)^l\nonumber\\
&=&\tau^2 \sum_{j=0}^n \sum_{l=0}^j (S_{\alpha,\beta}\star f)^l\label{eq2.4b}.
\end{eqnarray*}
Similarly, by \eqref{EqConvention}, it is easy to prove that
\begin{equation*}\label{eq2.4c}
\tau \sum_{j=0}^n k^{2-\alpha}_\tau(n-j)(k^\alpha_\tau\star S_{\alpha,1}\star f)^{j-1}=\tau^2 \sum_{j=0}^{n-1} \sum_{l=0}^j (S_{\alpha,\beta}\star f)^l,
\end{equation*}
and
\begin{equation*}\label{eq2.4d}
\tau \sum_{j=0}^n k^{2-\alpha}_\tau(n-j)(k^\alpha_\tau\star S_{\alpha,1}\star f)^{j-2}=\tau^2 \sum_{j=0}^{n-2} \sum_{l=0}^j (S_{\alpha,\beta}\star f)^l,
\end{equation*}
for all $n\geq 2.$

On the other hand,
\begin{equation}\label{eq2.4e}
_C\Na^{\alpha-1} f^n=\Na_\tau^{-(1-(\alpha-1))}(\Na_\tau^1 f)^n=\Na_\tau^{2-\alpha}(\Na^1_\tau f)^n=\tau \sum_{j=0}^n k^{2-\alpha}_\tau(n-j)(\Na^1 f)^j
\end{equation}
and, by \eqref{eq2.4a}--\eqref{eq2.4e}, we conclude that
\begin{eqnarray*}
_C\Na^\alpha (\tau ( {S}_{\alpha,1}\star f)^n)&=&\tau \sum_{j=0}^n k^{2-\alpha}_\tau(n-j)\Na^1_\tau(f)^j+\frac{A}{\tau^2}\tau \sum_{j=0}^n k^{2-\alpha}_\tau(n-j)\bigg[\tau^2 \sum_{j=0}^n \sum_{l=0}^j (S_{\alpha,\beta}\star f)^l\\
&&\quad -2\tau^2 \sum_{j=0}^{n-1} \sum_{l=0}^j (S_{\alpha,\beta}\star f)^l+\tau^2 \sum_{j=0}^{n-2} \sum_{l=0}^j (S_{\alpha,\beta}\star f)^l\bigg]\\
%&=&_C\Na^{\alpha-1} f^n+A\sum_{j=0}^n({S}_{\alpha,1}\star f)^j\\
&=&_C\Na^{\alpha-1} f^n+A(\tau ({S}_{\alpha,1}\star f)^n),
\end{eqnarray*}
for all $n\geq 2.$ We conclude that if $u^n:={S}^n_{\alpha,1}x_0+\tau (S_{\alpha,1}\star f)^n$ for $n\geq 2,$  then
\begin{eqnarray*}
_C\Na^\alpha(u^n)&=&\,_C\Na^\alpha\left({S}^n_{\alpha,1}x_0+\tau (S_{\alpha,1}\star f)^n\right)\\
&=&A{S}^n_{\alpha,1}x_0+A(\tau ({S}_{\alpha,1}\star f)^n)+_C\Na^{\alpha-1} f^n\\
&=&Au^n+_C\Na^{\alpha-1} f^n,
\end{eqnarray*}
for all $n\geq 2,$ that is, $u^n$ solves the equation
\begin{equation*}
_C\Na^\alpha u^n=Au^n+_C\Na^{\alpha-1} f^n,\quad n\geq 2.
\end{equation*}
We conclude that the sequence $(u^n)_{n\in \N_0}$ defined by
\begin{eqnarray*}
 u^n:=\left\{ \begin{array}{ccc}
{S}^n_{\alpha,1}x_0+\tau (S_{\alpha,1}\star f)^n, \quad n\geq 2,\\
x_0,\quad n=1, \\
0,\quad n=0, \\
\end{array} \right.
\end{eqnarray*}
solves the problem \eqref{MainEquation}. The uniqueness, follows from Proposition \ref{PropUniqueness}.
\end{proof}

%%%%%%%%%%%%%%%%%%%%%%%%%%%%%%%%%%%%%%%%%%%%%%%%%%%%%%%%%%%%%%%%%%%%%%%%%%%%%%%%%%%%%%%%%%%%%%%%%%%%%%%%%%%%%%%%%%%%%%%%%%%%%%%%%%%%%%%%%%%%%%%%%%%%%%%%%%%%%%%%%%%%%%%%%%%%%%%
%%%%%%%%%%%%%%%%%%%%%%%%%%%%%%%%%%%%%%%%%%%%%%%%%%%%%%%%%%%%%%%%%%%%%%%%%%%%%%%%%%%%%%%%%%%%%%%%%%%%%%%%%%%%%%%%%%%%%%%%%%%%%%%%%%%%%%%%%%%%%%%%%%%%%%%%%%%%%%%%%%%%%%%%%%%%%%%
%%%%%%%%%%%%%%%%%%%%%%%%%%%%%%%%%%%%%%%%%%%%%%%%%%%%%%%%%%%%%%%%%%%%%%%%%%%%%%%%%%%%%%%%%%%%%%%%%%%%%%%%%%%%%%%%%%%%%%%%%%%%%%%%%%%%%%%%%%%%%%%%%%%%%%%%%%%%%%%%%%%%%%%%%%%%%%%


\begin{thebibliography}{99}

\bibitem{Ab-Al-Di-21} L. Abad\'ias, E. \'Alvarez, S. D\'iaz, {\em Subordination principle, Wright functions and large-time behaviour for the discrete in time fractional diffusion equation.} arXiv:2102.10105v2.

\bibitem{Ab-Li-16} L. Abad\'ias, C. Lizama, {\em Almost automorphic mild solutions to fractional partial difference-differential equations,} Applicable Analysis, 95 (6) (2016), 1347-1369.


\bibitem{Ab-Li-Mi-Ve-19} L. Abadias, C. Lizama, P. J. Miana, M. P. Velasco, {\em On well-posedness of vector-valued fractional differential-difference equations,} Discrete and Continuous Dynamical Systems, Series A, {\bf 39} (5) (2019), 2679-2708.

\bibitem{Ab-11} T. Abdeljawad, {\em On Riemann and Caputo fractional differences,} Comput. Math. Appl. {\bf 62} (2011), no. 3, 1602-6111.

\bibitem{Ag-Cu-Li-14} R. Agarwal, C. Cuevas, C. Lizama, {\em Regularity of difference equations on Banach spaces, } Springer-Verlag, Cham, 2014. Hardcover ISBN 978-3-319-06446-8.

\bibitem{Al-Ca-Va-16} M. Allen, L. Caffarelli, A. Vasseur, {\em A parabolic problem with a fractional time derivative,} Arch. Ration. Mech. Anal. {\bf 221} (2016) 603-630.

\bibitem{Al-Di-Li-20} E. \'Alvarez, S. D\'iaz, C. Lizama, {\em C-Semigroups, subordination principle and the L\'evy $\alpha$-stable distribution on discrete time,} Comm. in Contemporary Mathematics, (2020) 2050063 (32 pages).


\bibitem{Ar-Li-08} D. Araya, C. Lizama, {\em Almost automorphic mild solutions to fractional differential equations,} Nonlinear Anal. {\bf 69} (2008), 3692-3705.%4

\bibitem{At-El-09}  F. Atici, P. Eloe, {\em Initial value problems in discrete fractional calculus,}
Proc. Amer. Math. Soc. {\bf 137} (2009), no. 3, 981-989.


\bibitem{Ba-01} E. Bazhlekova, {\em Fractional evolution equations in Banach spaces,} Ph.D. thesis, Eindhoven University of Technology, 2001.


\bibitem{Ca-Pl-15} P. de Carvalho-Neto,  G. Planas, {\em Mild solutions to the time fractional Navier-Stokes equations in $\mathbb{R}^N,$} J. Differential Equations {\bf 259} (2015), 2948-2980.


\bibitem{Cu-07} E. Cuesta, {\em Asymptotic behaviour of the solutions of fractional integro-differential equations and some time discretizations,} Discrete Contin. Dyn. Syst. 2007, Dyn. Syst. and Diff. Eqns. Proc. of the 6th AIMS Int. Conference, suppl., 277-285.

\bibitem{Cu-Li-08} C. Cuevas, C. Lizama, {\em Almost automorphic solutions to a class of semilinear fractional differential equations,}  Appl. Math. Lett. \textbf{21} (2008), 1315-1319.

\bibitem{Cu-So-09} C. Cuevas, J. de Souza, {\em $S$-asymptotically $\omega$-periodic solutions of semilinear fractional integro-differential equations,}  Appl. Math. Lett. \textbf{22} (2009), 865-870.

\bibitem{En-Na-00} K. Engel, R. Nagel, {\em One-parameter semigroups for linear evolution equations.} GTM vol. {\bf 194}, 2000.

\bibitem{Fe-12} R. Ferreira, {\em Discrete fractional Gronwall inequality,} Proc. Amer. Math. Soc. {\bf 140} (2012) (5), pp. 1605-1612.

\bibitem{Go-11} C. Goodrich, {\em Existence and uniqueness of solutions to a fractional difference equation with nonlocal conditions,} Comput. Math. Appl. 61 (2011), no. 2, 191-202.

\bibitem{Go-Li-20a} C. Goodrich, C. Lizama, {\em A transference principle for nonlocal operators using a convolutional approach: Fractional monotonicity and convexity,} Israel J. of Mathematics, {\bf 236} (2020), 533-589.

\bibitem{Go-Li-20b} C. Goodrich, C. Lizama, {\em Positivity, monotonicity and convexity for convolution operators,} Discrete and Continuous Dynamical Systems, Series A, 2020, {\bf 40} (8), 4961-4983.

\bibitem{Go-Pe-15} C. Goodrich, A. Peterson, {\em Discrete fractional calculus,} Springer, Cham, 2015.

\bibitem{Gr-Ry-00} I. Gradshteyn, I. Ryzhik, {\em Table of integrals, series and products,} Academic Press, New York, 2000.


\bibitem{Haa-06} M. Haase, {\em The functional calculus for sectorial operators,} Operator Theory: Advances and applications, 169, Birk\"auser Verlag, Basel, 2006.

\bibitem{He-Me-Po-21} H. Henr\'iquez, J. G. Mesquita, J. C. Pozo, {\em Existence of solutions of the abstract Cauchy problem of fractional order,} J. of Functional Analysis, 2021, to appear.


\bibitem{Ji-La-Zh-16} B. Jin, R. Lazarov, Z. Zhou, {\em Two fully discrete schemes for fractional diffusion and diffusion-wave equations with nonsmooth data,} SIAM J. Sci. Comput. {\bf 38} (2016), no. 1, A146-A170.

\bibitem{Ji-La-Zh-19} B. Jin, R. Lazarov, Z. Zhou, {\em Numerical methods for time-fractional evolution equations with nonsmooth data: a concise overview,} Comput. Methods Appl. Mech. Engrg. {\bf 346} (2019), 332-358.

\bibitem{Ji-Li-Zh-18} B. Jin, B. Li, Z. Zhou, {\em Discrete maximal regularity of time-stepping schemes for fractional evolution equations,} Numer. Math. {\bf 138} (2018), no. 1, 101--131.

\bibitem{Ji-Li-Zh-18b} B. Jin, B. Li, Z. Zhou, {\em Numerical analysis of nonlinear subdiffusion equations,} SIAM J. Numer. Anal. {\bf 56} (2018), no. 1, 1-23.

\bibitem{Ji-Li-Zh-19} B. Jin, B. Li, Z. Zhou, {\em Subdiffusion with a time-dependent coefficient: analysis and numerical solution,} Math. Comp. {\bf 88} (2019), no. 319, 2157-2186.

\bibitem{Ke-Li-Wa-13b} V. Keyantuo, C. Lizama, M. Warma, {\em Spectral criteria for solvability of boundary value problems and positivity of solutions of time-fractional differential equations,} Abstr. Appl. Anal. 2013, Art. ID 614328, 11 pp.

\bibitem{Ki-Sr-Tr-06} A. Kilbas, H. Srivastava, J. Trujillo, {\em Theory and applications of fractional differential equations,} North-Holland Mathematics studies 204, Elsevier Science B.V., Amsterdam, 2006. %20

\bibitem{Ku-57} B. Kuttner, {\em On differences of fractional order,} Proc. London Math. Soc. {\bf 3}, No 1 (1957), 453-466.

\bibitem{Li-Ch-Li-10} M. Li, C. Chen, F. Li, {\em On fractional powers of generators of fractional resolvent families,} J. Funct. Anal. {\bf 259} (2010) 2702-2726.


\bibitem{Li-Pe-Ji-12} K. Li, J. Peng, J. Jia, {\em Cauchy problems for fractional differential equations with Riemann-Liouville fractional derivatives,} J. Funct. Anal. {\bf 263} (2012), 476-510.

\bibitem{Li-Li-15} Z. Liu, X. Li, {\em Approximate controllability of fractional evolution systems with Riemann-Liouville fractional derivatives,} SIAM J. Control Optim. {\bf 53} (2015), 1920-1933.


\bibitem{Li-17} C. Lizama, {\em The Poisson distribution, abstract fractional difference equations, and stability,} Proc. Amer. Math. Soc. {\bf 145} (2017), no. 9, 3809--3827.

\bibitem{Li-He-Zh-20} C. Lizama, W. He, Y. Zhou, {\em The Cauchy problem for discrete-time fractional evolution equations,} Journal of Computational and Applied Mathematics, {\bf 370} (2020), 112683.


\bibitem{Li-Ng-13}  C. Lizama, G. M. N'Gu\'er\'ekata, {\em Mild solutions for abstract fractional differential equations,}  Appl. Anal. \textbf{ 92} (2013), 1731-1754.

\bibitem{Li-FPo-12} C. Lizama, F. Poblete. {\em  On a functional equation associated with $(a,k)$-regularized resolvent families,}  Abstr. Appl. Anal. 2012, Art. ID 495487, 23 pp.


\bibitem{Lu-86}  Ch. Lubich, {\em Discretize fractional calculus,} SIAM J. Math. Anal., {\bf 17} (1986), 704-719.


\bibitem{Mi-Ro-93}  K. Miller, B. Ross, {\em An Introduction to the fractional calculus and fractional differential equations,} Wiley, New York 1993.%24


\bibitem{Pon-20a} R. Ponce, {\em Asymptotic behavior of mild solutions to fractional Cauchy problems in Banach spaces,} Appl. Math. Lett. {\bf 105} (2020), 106322.

\bibitem{Pon-20c} R. Ponce, {\em Subordination Principle for fractional diffusion-wave of Sobolev type,} Fract. Calc. Appl. Anal. 23 (2020), no. 2, 427-449.

\bibitem{Pon-19} R. Ponce, {\em Time discretization of fractional subdiffusion equations via fractional resolvent operators.} Comput. Math. Appl. {\bf 80} (2020), no. 4, 69-92.


\bibitem{Pr} J. Pr\"uss. {\em Evolutionary Integral Equations and Applications}. Monographs Math.,  {\bf 87}, Birkh\"{a}user Verlag, 1993. %26

\bibitem{To-Ta-17} B. Torebek, R. Tapdigoglu, {\em Some inverse problems for the nonlocal heat equation with Caputo fractional derivative,} Math. Methods Appl. Sci. {\bf 40} (2017), no. 18, 6468-6479.


\bibitem{Wa-Ch-Xi-12} R. Wang, D. Chen, T. Xiao, {\em Abstract fractional Cauchy problems with almost sectorial operators,} J. Diff. Equations 252 (2012), 202-235.

\bibitem{Xi-Wa-18} Z. Xia, D. Wang, {\em Asymptotic behavior of mild solutions for nonlinear fractional difference equations,}  Fractional Calculus and Applied Analysis, 2018, DOI: 10.1515/fca-2018-0029.



%\bibitem{Ar-Ba-Hi-Ne-11} W. Arendt, C. Batty, M. Hieber, F. Neubrander, {\em Vector-Valued Laplace transforms and Cauchy problems.} Monogr. Math., vol. \textbf{ 96}, Birkh\"auser,    Basel, 2011.

%\bibitem{Po-99} I. Podlubny, {\em Fractional differential equations,} Academic Press, San Diego, 1999.









\end{thebibliography}
\end{document}